\documentclass[reqno]{amsart}
\usepackage{amssymb,amsmath,amsfonts,bm}
\usepackage{dsfont}
\usepackage{epsfig}
\usepackage{graphicx}
\usepackage{enumerate}
\usepackage{verbatim}
\usepackage{color}
\usepackage{hyperref}
\usepackage{caption}
\usepackage{geometry}
\date{} 

 \DeclareMathOperator\supp{supp}

\DeclareMathOperator\inode{in}
\DeclareMathOperator\enode{end}
\theoremstyle{plain}
\newtheorem{theorem}{Theorem}[section]
\newtheorem{definition}[theorem]{Definition}
\newtheorem{proposition}[theorem]{Proposition}

\newtheorem{lemma}[theorem]{Lemma}
\newtheorem{corollary}[theorem]{Corollary}
\newtheorem{example}[theorem]{Example}
\newtheorem{remark}[theorem]{Remark}

\begin{document}

\parskip=8pt

\title[On the $\ell^1$ non-embedding in the James Tree Space]
{On the $\ell^1$ non-embedding in the James Tree Space}
\author[Ioakeim Ampatzoglou]{Ioakeim Ampatzoglou}
\address{Ioakeim Ampatzoglou,  
Department of Mathematics, The University of Texas at Austin.}
\email{ioakampa@math.utexas.edu}

\vspace*{-3cm}
\maketitle
\begin{abstract}
 James Tree Space ($\mathcal{JT}$), introduced by R. James in \cite{james tree},  is the first Banach space constructed having non-separable conjugate and not containing $\ell^1$. James actually proved that every infinite dimensional subspace of $\mathcal{JT}$ contains a Hilbert space, which implies the $\ell^1$ non-embedding. In this expository article, we present a direct proof of the $\ell^1$ non-embedding, using Rosenthal's $\ell^1$- Theorem \cite{rosenthal} and some measure theoretic arguments, namely Riesz's Representation Theorem \cite{rudin}.
\end{abstract}
\section{Introduction}
For many years, the conjecture that a separable Banach space with non-separable conjugate will contain $\ell^1$, up to embedding, was an open question in Banach space theory. It is well-known that $(\ell^1)^*$ coincides with $\ell^\infty$, which is non-separable, so a natural question is whether all separable Banach spaces with this property "look like" $\ell^1$, up to embedding. This conjecture was proved false by R. James \cite{james tree} who made a ingenius construction call the James Tree Space. The main idea relies on previous work of R. James \cite{james}, where a quasi-reflexive separable Banach space, isometric to its second conjugate and Hilbert-saturated, was constructed. By Hilbert-saturated we mean that each infinite dimensional subspace contains a Hilbert space, up to embedding. This space, called the James space, clearly has a separable conjugate though. R. James was able to preserve the Hilbert-saturation property but remove separability of the conjugate by creating a binary tree structure where intuitively each infinite branch of the tree will be basis for a James space. The $\ell^1$ non-embedding then follows as an immediate consequence of the Hilbert-saturation property. In this work, we provide a direct proof of the $\ell^1$ non-embedding, which was the main part of the conjecture, in a direct way i.e. without proving Hilbert-saturation property. For this purpose, we first review the Schauder bases theory, which is crucial for this construction and  introduce the James Tree Space. We then  apply Riesz's Representation Theorem \cite{rudin} to an appropriate $w^*$-compact subset of the conjugate space and Rosenthal's $\ell^1$-Theorem \cite{rosenthal} yields the claim.
\section{Preliminaries}
In this preliminary section, we summarize some major results from Schauder bases in Banach spaces which will be useful throughout this paper. The majority of the proofs can be found in most textbooks of Banach Space Theory, hence we provide proofs only for the results  which are not exactly stated in the bibliography as they are stated here. Our approach is based on \cite{linderstrauss}. 

\subsection{Notation} Let us clarify the notation used. Throughout this paper, all Banach spaces considered are assumed infinite dimensional unless stated. Given a Banach space $X$ we denote its unit ball by $B_X$ and its conjugate spaces by $X^*, X^{**}$, etc. Given a sequence $\{x_n\}_{n\in\mathbb{N}}$ in $X$ we denote $\langle x_n:n\in\mathbb{N}\rangle$ to be the vector space spanned by this sequence and $[x_n:n\in\mathbb{N}]$ to be the closure of the space spanned with respect to the norm. Finally. given a Banach space $X$ we denote $\wedge:X\to X^{**}$ the canonical embedding of $X$ in $X^{**}$ given by
$$\widehat{x}(x^*)=x^*(x),\quad\forall x\in X.$$
Recall that the canonical embedding is a linear isometry and a Banach space is called reflexive if the canonical embedding is surjective.
\subsection{Definition of Schauder basis}
As known a Banach space necessarily has uncountable algebraic basis. However, in most reasonable separable Banach spaces, we are able to find a countable topological basis i.e. each element can be expanded as a series with respect to the norm. More precisely we give the following definition:
\begin{definition}
Let X be a Banach space and  $\left\{e_n\right\}_{n\in\mathbb{N}}$ be
 a sequence of distinct elements of $X$ . The sequence 
$\left\{e_n\right\}_{n\in\mathbb{N}}$
is called a Schauder basis, or just a basis, of $X$ if for any $x\in{X}$ 
there is a unique sequence 
$\left\{\lambda_n\right\}_{n\in\mathbb{N}}\subseteq\mathbb{R}$
 such that
\begin{equation*}
x=\sum_{n=1}^\infty\lambda_ne_n .
\end{equation*}
\end{definition}
The existence of Schauder basis is easily seen to be possible only in separable Banach spaces. However, the converse is not true, as shown by P. Enflo in \cite{enflo}.
\begin{proposition}
Let $X$ be a Banach space with a Schauder basis. Then $X$ is separable.
\end{proposition}

Consider $X$ to be a Banach space with  Schauder basis $\{e_n\}_{n\in\mathbb{N}}$. For any $n\in\mathbb{N}$ we define $e_n^{*}:X\rightarrow\mathbb{R}$
by:
\begin{equation*}
e_n^{*}(x)=e_n^{*}(\sum_{k=1}^\infty\lambda_ke_k)=\lambda_n.
\end{equation*}
Clearly $x=\displaystyle\sum_{n=1}^\infty e_n^{*}(x)e_n$.
One can easily prove the following:
\begin{proposition}\label{biothogonal boundness}
For any $n\in\mathbb{N}$,  $e_n^{*}\in X^*$ i.e. $e_n^*$ 
is a bounded linear functional.
\end{proposition}

The functionals $\left\{e_n^{*}\right\}_{n\in\mathbb{N}}$ are called biorthogonal functionals of the basis
 $\left\{e_n\right\}_{n\in\mathbb{N}}$. 
 
We now state an equivalent characterization of  Schauder bases. In fact, this is how one usually checks whether a given sequence in a Banach space is a Schauder basis.
\begin{proposition}\label{equiv schauder}
Let $X$ be a Banach space and a sequence $\left\{e_n\right\}_{n\in\mathbb{N}}$ of pairwise distinct, non-zero elements of $X$.
 The following statements are equivalent:
 \begin{enumerate}[(i)]
\item $\left\{e_n\right\}_{n\in\mathbb{N}}$ 
is a Schauder basis of  $X$
\item The following hold: \begin{itemize}
\item$X=[e_n:n\in\mathbb{N}].$ 
\item $\exists K>0$
 such that for any
$m>n\in\mathbb{N}$  and
$\lambda_1,...,\lambda_m\in\mathbb{R}$, there holds
\begin{equation}\label{constant}
||\sum_{i=1}^n\lambda_ie_i||\leq K||\sum_{i=1}^m\lambda_ie_i||.
\end{equation}
\end{itemize}
\end{enumerate}
\end{proposition}

\begin{remark}
The infimum number $K$ in \eqref{constant} is called constant of the basis. In the special case where this constant is unit, the basis is called monotone. It is clear that, for a basis to be monotone, condition \eqref{constant} reduces to showing that for any $n\in\mathbb{N}$ and $\lambda_1,...,\lambda_{n+1}\in\mathbb{R}$, there holds
$$\|\sum_{i=1}^n\lambda_ie_i\|\leq\|\sum_{i=1}^{n+1}\lambda_ie_i\|.$$
\end{remark}
\begin{example}
Let $1\leq p<\infty$. Recall by $\ell^p$ we denote the Banach spaces
$$\ell^p=\left\{x=\{x_n\}_{n\in\mathbb{N}}:\sum_{n=1}^\infty |x_n|^p<\infty\right\},$$
with norm
$$\|x\|_p= (\sum_{n=1}^\infty |x_n|^p)^{1/p}.$$

Then the sequence $\{e_n\}_{n\in\mathbb{N}}$, given by $e_n=(0,0,...,1,0,...)$
(n-position) is a monotone Schauder basis of $\ell^p$.
We call it the natural basis of $\ell^p$.
\end{example}
\begin{proof}
Let $x=(x_1,x_2...,x_n,...)\in\ell^p$.
 Define $s_n=\displaystyle\sum_{i=1}^nx_ie_i$. Then we get
 \begin{equation*}
 ||s_n-x||=(\sum_{i=n+1}^\infty|x_i|^p)^{1/p}
 \overset{n\rightarrow\infty}{\longrightarrow}0.
 \end{equation*}
 Moreover, for all $n\in\mathbb{N}$ and $\lambda_1,...,\lambda_{n+1}\in\mathbb{R}$,
 we have
 \begin{equation*}
  ||\sum_{i=1}^n\lambda_ie_i||=(\sum_{i=1}^n|\lambda_i|^p)^{1/p}
 \leq(\sum_{i=1}^{n+1}|\lambda_i|^p)^{1/p}=||\sum_{i=1}^{n+1}\lambda_ie_i||,
 \end{equation*}
 and the claim follows. \end{proof}
 
We now deduce some useful criterias for $w^*$-convergence in Banach spaces with Schauder bases.
\begin{proposition}\label{prop 1.1.10}
Let $X$ be a Banach space,
$x^*\in X^*\setminus\{0\}$ and a bounded sequence $\left\{x_n^*\right\}_{n\in\mathbb{N}}$ in $X^*$.
Assume there is norm-dense $S\subseteq X$ such that $$x_n^*(s)\overset{n\to\infty}{\longrightarrow}x^*(s),\quad\forall s\in S.$$ Then  $x_n^*\overset{w^*}\to x^*$. 
\end{proposition}
\begin{proof}
The result is trivial if $x_n=0,\quad\forall n\in\mathbb{N}$. Therefore, since $\left\{x_n^*\right\}_{n\in\mathbb{N}}$ is bounded too, we may well assume that
$$0<M:=\sup_{n\in\mathbb{N}}\|x_n^*\|<\infty.$$
It suffices to show that $x^*(x)=\displaystyle\lim_{n\to\infty}x_n^*(x),\quad\forall x\in X$. Let $x\in X$. By density, there is a sequence $\left\{s_n\right\}_{n\in\mathbb{N}}\subseteq S$ with $x=\displaystyle\lim_{n\to\infty}s_n$. Let $\epsilon>0$ and consider $N,n_0\in\mathbb{N}$ such that
\begin{eqnarray*}
||s_N-x||<\displaystyle\frac{\epsilon}{3\max\left\{M,||x^*||\right\}}&\mbox{ and }&
|x_n^*(s_N)-x^*(s_N)|<\frac{\epsilon}{3},\quad\forall n\geq n_0.
\end{eqnarray*}
Then for any $n\geq n_0$, we have
\begin{eqnarray*}
|x_n^*(x)-x^*(x)|&\leq& |x_n^*(x)-x_n^*(s_N)|+|x_n^*(s_N)-x^*(s_N)|+||x^*(s_N)-x^*(x)|\\
&<&M||s_N-x||+\frac{\epsilon}{3}+\|x^*\|\|s_N-x\|
<\frac{\epsilon}{3}+\frac{\epsilon}{3}+\frac{\epsilon}{3}=\epsilon.
\end{eqnarray*}
The result is proved.\end{proof}
\begin{corollary}\label{cor 1.1.11}
Let $X$ be a Banach space with basis $\left\{e_n\right\}_{n\in\mathbb{N}}$. Consider
$x^*\in X^*$ and $\left\{x_k^*\right\}_{k\in\mathbb{N}}$ a bounded sequence in $X^*$ 
If $x_k^*(e_n)\overset{k\to\infty}{\longrightarrow}x^*(e_n)\quad\forall n\in\mathbb{N}$, then
$x_n^*\overset{w^*}\to x^*$. 
\end{corollary}
\begin{proof}
Defining $S=\langle e_n:n\in\mathbb{N}\rangle$. Then $S$ is dense in $X$. Linearity of $\left\{x_k^*\right\}_{k\in\mathbb{N}}$, $x^*$ and  Proposition \eqref{prop 1.1.10} yield the result.
\end{proof}
\subsection{Basic sequences, blocks and equivalence}
In this section we introduce the notion basic sequences in Banach spaces and some elementary type of basic sequences called blocks.
\begin{definition}

 Let $X$ be a Banach space and
$\left\{x_n\right\}_{n\in\mathbb{N}}$ a sequence of pairwise distinct, non-zero elements of $X$.
The sequence $\left\{x_n\right\}_{n\in\mathbb{N}}$ 
is called basic sequence if it is Schauder basis of the subspace
$[x_n:n\in\mathbb{N}]$.

 We define the constant of the basic sequence $\left\{x_n\right\}_{n\in\mathbb{N}}$ 
as the constant of the Schauder basis of
$[x_n:n\in\mathbb{N}]$. 

 Finally, we define its biorthogonal functionals  as $x_n^*:[x_n:n\in\mathbb{N}]\rightarrow\mathbb{R}$ given by $$x_n^*(\sum_{k=1}^\infty\lambda_kx_k)=\lambda_n.$$
Since $\{x_n\}_{n\in\mathbb{N}}$ is basic, Proposition \eqref{biothogonal boundness} implies $x_n^*\in[x_k:k\in\mathbb{N}]^*,\quad\forall n\in\mathbb{N}$.
\end{definition}
\begin{remark}
Hahn-Banach Theorem implies that for each $n\in\mathbb{N}$, $x_n^*$ can be extended to an element of $X^*$. Therefore, without loss of generality, we may assume that  $x_n^*\in X^*$, for each $n\in\mathbb{N}$.
\end{remark}

We immediately get the following characterization:
\begin{corollary}\label{equiv basic} Let $X$ be a Banach space and a sequence $\left\{e_n\right\}_{n\in\mathbb{N}}$ of pairwise distinct, non-zero elements of $X$.
 The following statements are equivalent:
 \begin{enumerate}[(i)]
\item $\left\{e_n\right\}_{n\in\mathbb{N}}$ 
is basic sequence.
\item $\exists K>0$
 such that for any
$m>n\in\mathbb{N}$  and
$\lambda_1,...,\lambda_m\in\mathbb{R}$, there holds
\begin{equation}
||\sum_{i=1}^n\lambda_ie_i||\leq K||\sum_{i=1}^m\lambda_ie_i||.
\end{equation}
\end{enumerate}
\end{corollary}

We now define the notion of equivalence of two basic sequences.
\begin{definition}
Let $X,Y$ be Banach spaces. Consider the sequences
$\left\{x_n\right\}_{n\in\mathbb{N}}\subseteq X$ and $
\left\{y_n\right\}_{n\in\mathbb{N}}\subseteq Y$. The sequences 
$\left\{x_n\right\}_{n\in\mathbb{N}}$ and
$\left\{y_n\right\}_{n\in\mathbb{N}}$ are called equivalent if there exist
$c,C>0$ such that for any
$n\in\mathbb{N}\mbox{ and } \lambda_1,...,\lambda_n\in
\mathbb{R}$,
there holds
\begin{equation*}
c||\sum_{i=1}^n\lambda_n x_n||\leq ||\sum_{i=1}^n\lambda_ny_n||
\leq C||\sum_{i=1}^n\lambda_nx_n||.
\end{equation*}
\end{definition}
\begin{remark}
It is clear that equivalence of sequences is an equivalence relation which preserves basic sequences.
\end{remark}
\begin{remark}\label{series equivalnce} Let $X,Y$ be Banach spaces and $\left\{x_n\right\}_{n\in\mathbb{N}}\subseteq X,
\left\{y_n\right\}_{n\in\mathbb{N}}\subseteq Y$ equivalent sequences. Then the series
$\sum_{n=1}^\infty\alpha_nx_n$ converges iff the series
$\sum_{n=1}^\infty\alpha_ny_n$ converges.
\end{remark}
\begin{proof}Let $\left\{\alpha_n\right\}_{n\in\mathbb{N}}$ such that
$\sum_{n=1}^\infty\alpha_nx_n$ converges. We will show that the sequence of partial sums
$\left\{\sum_{i=1}^n\alpha_iy_
i\right\}_{n\in\mathbb{N}}$ is Cauchy. Indeed, 
for any $\epsilon>0$ there is $N\in\mathbb{N}$
such that for all $m>n>N$, we have
\begin{equation*}
||\sum_{i=n+1}^m\alpha_ix_i||<\frac{\epsilon}{C}.
\end{equation*}
Then we get
\begin{equation*}
||\sum_{i=n+1}^m\alpha_iy_i||<\epsilon.
\end{equation*}
The other way is identical.
\end{proof}
Equivalent sequences can be characterized in the following equivalent ways:
\begin{proposition}\label{1.3.2}
Let $X,Y$ be Banach spaces. Consider a basic sequence 
$\left\{x_n\right\}_{n\in\mathbb{N}}\subseteq X$ and a sequence $
\left\{y_n\right\}_{n\in\mathbb{N}}\subseteq Y$. Then the following are equivalent:
\begin{enumerate}[(i)]
 \item $\left\{x_n\right\}_{n\in\mathbb{N}}$
and $\left\{y_n\right\}_{n\in\mathbb{N}}$
are equivalent.
\item
  There is an isomorphism
$T:[x_n:n\in\mathbb{N}]\rightarrow [y_n:n\in\mathbb{N}]$, with
$T(x_n)=y_n,\quad\forall n\in\mathbb{N}$.
\end{enumerate}
\end{proposition}
\begin{proof}
$(i)\Rightarrow (ii)$ 
We define the mapping $T:[x_n:n\in\mathbb{N}]\rightarrow [y_n:n\in\mathbb{N}]$ by
$$T(x)=T(\displaystyle\sum_{n=1}^\infty\alpha_nx_n)=
\displaystyle\sum_{n=1}^\infty\alpha_ny_n.$$ The mapping $T$ is well-defined and linear due to the fact that 
the sequence $\left\{x_n\right\}_{n\in\mathbb{N}}$ is basic and Remark \ref{series equivalnce}.  We first show that $T$ is bounded.
For any $n\in\mathbb{N}$, let us define the linear operators $T_n:[x_n:n\in\mathbb{N}]\rightarrow \langle y_1,...,y_n\rangle$, $F_n:\langle x_1,...,x_n\rangle\rightarrow \langle y_1,...,y_n\rangle$ and $P_n:[x_n:n\in\mathbb{N}]\to\langle x_1,...,x_n\rangle$, given respectively by
$$T_n(\sum_{i=1}^\infty\alpha_ix_i)=\sum_{i=1}^n\alpha_iy_i,$$
 $$F_n(\sum_{i=1}^n\alpha_ix_i)=\sum_{i=1}^n\alpha_iy_i,$$
 $$P_n(\sum_{i=1}^\infty\alpha_ix_i)=\sum_{i=1}^n\alpha_ix_i.$$
Notice that $F_n$ is bounded since its domain is finite dimensional. Moreover, since $\{x_n\}_{n\in\mathbb{N}}$ is basic, Corollary \eqref{equiv basic} implies that for any $x=\sum_{i=1}^\infty\lambda_ix_i\in X$, we have
$$P_n(x)=\|\sum_{i=1}^n\lambda_ix_i\|\leq K\|\sum_{i=1}^m\lambda_ix_i\|,\quad\forall m>n,$$
where $K$ is the constant of $\{x_n\}_{n\in\mathbb{N}}$.
Letting $m\to\infty$ we get $\|P_n(x)\|\leq K\|x\|$, hence $P_n$ is bounded. But $T_n=F_n\circ P_n$, so $T_n$ is bounded for any $n\in\mathbb{N}$.
  By definition of $T_n$, the following point-wise convergence holds:
 $$T(x)=\displaystyle\lim_{n\to \infty} T_n(x)
,\quad\forall x\in X.$$
Therefore, Banach-Steinhauss Theorem implies $T$ is bounded. We will also show that $T$ is a bijection and the result will come by the Open Mapping Theorem. Since $\{y_n\}_{n\in\mathbb{N}}$ is equivalent to $\{x_n\}_{n\in\mathbb{N}}$, it is basic as well. Therefore each $y\in[y_n:n\in\mathbb{N}]$ can be uniquely written as $y=\sum_{n=1}^\infty y_n$. This directly implies that $T$ is a bijection, since $T(x_n)=y_n,\quad\forall n\in\mathbb{N}$ and $T$ is bounded.

$(ii)\Rightarrow (i)$ Comes  immediately for $c=\displaystyle\frac{1}{||T^{-1}||}$ and
$C=||T||$.
\end{proof}
\begin{remark}\label{1.3.3}
Let $X$ be a Banach space with basis $\left\{x_n\right\}_{n\in\mathbb{N}}$ 
and $Y$ a Banach space. If there is a sequence
$\left\{y_n\right\}_{n\in\mathbb{N}}$ 
in $Y$ and $c,C>0$ such that for any
$n\in\mathbb{N}$ and $\lambda_1,...,\lambda_n\in
\mathbb{R}$
there holds
\begin{equation*}
c||\sum_{i=1}^n\lambda_nx_n||\leq ||\sum_{i=1}^n\lambda_ny_n||
\leq C||\sum_{i=1}^n\lambda_nx_n||.
\end{equation*}
Then $X$ embeds isomorphically in $Y$.
\end{remark}
\begin{proof}
By  Proposition \ref{1.3.2}, it is enough to show that the sequence $\left\{y_n\right\}_{n\in\mathbb{N}}$ is basic. Let $K$ be the constant of $\{x_n\}_{n\in\mathbb{N}}$.
Consider $m>n\in\mathbb{N}$ and $\lambda_1,...\lambda_m
 \in\mathbb{R}$. Then
 \begin{equation*}
 ||\sum_{i=1}^n\lambda_iy_i||\leq CK||\sum_{i=1}^m\lambda_ix_i||
 \leq\frac{CK}{c}||\sum_{i=1}^m\lambda_iy_i||,
 \end{equation*}
 and the claim is proved.
\end{proof}

We now restrict our attention to a specific class of basic sequences, called blocks. Blocks are much easier to handle than arbitrary basic sequences.
\begin{definition}
Let $X$ be a Banach space with basis $\left\{e_n\right\}_{n\in\mathbb{N}}$. 
A sequence 
$\left\{u_n\right\}_{n\in\mathbb{N}}$  of pairwise distinct non-zero elements of $X$ will be called block of $\left\{e_n\right\}_{n\in\mathbb{N}}$ if there is a sequence of real numbers
$\left\{\alpha_i\right\}_{i\in\mathbb{N}}$ and an increasing sequence
 $\left\{n_i\right\}_{i\in\mathbb{N}}$ of positive integers such that for any $k\in\mathbb{N}$, there holds
\begin{equation*}
u_k=\sum_{i=n_k+1}^{n_{k+1}}\alpha_ie_i
\end{equation*}
\end{definition}
\begin{remark}
It is worth mentioning that the notion of a block sequence is always defined with respect to a given Schauder basis.
\end{remark}
It is straightforward that blocks are basic sequences.
\begin{proposition}
Let $X$ be a Banach space with basis $\left\{e_n\right\}_{n\in\mathbb{N}}$
and constant $K$. Then every block sequence is basic of constant less or equal than $K$.
\end{proposition}
\begin{proof}
For all $m>k\in\mathbb{N}$ and
$\lambda_1,...,\lambda_m\in\mathbb{R}$ we have
\begin{eqnarray*}
||\sum_{j=1}^k\lambda_ju_j||&=&||\sum_{j=1}^k\lambda_j\sum_{i=n_j+1}^{n_{j+1}}\alpha_ie_i||
=||\sum_{j=1}^k\sum_{i=n_j+1}^{n_{j+1}}\lambda_j\alpha_ie_i||\\
&\leq&K||\sum_{j=1}^m\sum_{i=n_j+1}^{n_{j+1}}\lambda_j\alpha_ie_i||\leq K||\sum_{j=1}^m\lambda_ju_j||,
\end{eqnarray*}
so the sequence $\left\{u_n\right\}_{n\in\mathbb{N}}$ is basic with constant less or equal than
  $K$.
 \end{proof}
 The following result gives some sufficient conditions for a basic sequence to be equivalent to a block up to subsequence. It is usually referred in literature as sliding hump argument.
 \begin{lemma} (Sliding Hump Argument)
Let $X$ be a Banach space with basis $\left\{e_n\right\}_{n\in\mathbb{N}}$ and a sequence $\left\{x_n
\right\}_{n\in\mathbb{N}}$ in $X$ such that 
\begin{itemize}
\item $\displaystyle\inf_{n\in\mathbb{N}}||x_n||>0.$
\item $\displaystyle\lim_{n\to\infty}e_k^*(x_n)=0,\quad\forall k\in\mathbb{N}.$
\end{itemize}
Then there is a subsequence $\left\{x_n'\right\}_{n\in\mathbb{N}}$ of $\left\{x_n\right\}_{n\in\mathbb{N}}$ which  equivalent to a block of $\left\{e_n\right\}_{n\in\mathbb{N}}$.
\end{lemma}

We prove the following useful result about embeddings.

\begin{proposition}\label{prop 1.4.5}
Let $X$ be a Banach space with Schauder basis $\left\{e_n\right\}_{n\in\mathbb{N}}$. Assume $Y$ is a Banach space which embedds isomorphically in $X$, with Schauder basis $\left\{y_n\right\}_{n\in\mathbb{N}}$, satisfying
$$0<m\leq\inf_{n\in\mathbb{N}}\|y_n\|\leq\sup_{n\in\mathbb{N}}\|y_n\|\leq M<\infty.$$
 Then there 
 is a block of $\left\{e_n\right\}_{n\in\mathbb{N}}$ equivalent to  $\left\{y_n\right\}_{n\in\mathbb{N}}$.
\end{proposition}
\begin{proof}
Since $Y$ embedds in $X$, there is a sequence
$\left\{x_n\right\}_{n\in\mathbb{N}}$ in $X$ which is basic and equivalent to $\left\{y_n\right\}_{n\in\mathbb{N}}$. Therefore, there are $m',M'>0$ such that
$$m'\leq||x_n||\leq M'\quad\forall n\in\mathbb{N}.$$ For any $k\in\mathbb{N}$,
we have $|e_k^*(x_n)|\leq||e_k^*||M,\quad\forall n\in\mathbb{N}$,
so the real numbers sequence $\left\{e_k^*(x_n)\right\}_{n\in\mathbb{N}}$ is bounded for all $k\in\mathbb{N}$.
Therefore, with a diagonal argument, we may construct a subsequence $\left\{x_n'\right\}_{n\in\mathbb{N}}$ of $\left\{x_n\right\}_{n\in\mathbb{N}}$ such  that the sequence $\left\{e_k^*(x_n')\right\}_{n\in\mathbb{N}}$ converges for all $k\in\mathbb{N}$. Let us denote $z_n=x_{n+1}'-x_n'$.
For any $k\in\mathbb{N}$, we clearly have
$$\displaystyle\lim_{n\to\infty} e_k^*(z_n)=0.$$ Using a sliding hump argument, we may find a subsequence $\left\{z_n'\right\}_{n\in\mathbb{N}}$ and a block $\left\{u_n\right\}_{n\in\mathbb{N}}$ of $\left\{x_n\right\}_{n\in\mathbb{N}}$ which are equivalent. Clearly $\left\{u_n\right\}_{n\in\mathbb{N}}$
is equivalent to $\left\{y_n\right\}_{n\in\mathbb{N}}$.
\end{proof}

We finally mention, without proof, a very important Theorem which will turn out to be essential for our exposition. As known,  $\ell^1$ cannot embedd in a space with separable dual since its dual coincides with $\ell^{\infty}$.  H. Rosenthal proved a partial inverse of this, the famous $\ell^1$-Theorem \cite{rosenthal}.
\begin{theorem}($\ell^1$-Theorem)
Let $X$ be a Banach space and $\left\{x_n\right\}_{n\in\mathbb{N}}$ a bounded sequence in $X$.
Then there holds exclusively one of the following:\\
i) The sequence $\left\{x_n\right\}_{n\in\mathbb{N}}$ has weak-Cauchy subsequence.\\
ii) The sequence $\left\{x_n\right\}_{n\in\mathbb{N}}$ is basic and equivalent to the standard basis of $\ell^1$.
\end{theorem}
The one direction is immediate since the standard basis of $ \ell^1$  
cannot have a weak-Cauchy subsequence. The other direction is much more complicated though. Reader can find more in \cite{rosenthal} .
As an application of Rosenthal's Theorem we prove the following useful Proposition:
\begin{proposition}\label{1.3.7}
Let $X$ a Banach space not containing
$\ell^1$.
 Then every infinite dimensional subspace of $X$ has a weak-Cauchy unitary sequence.
 \end{proposition}

\begin{proof}
Let $Y$ be infinite dimensional subspace of $X$. Then $B_Y$ is not norm-compact.
Therefore there is $\left\{s_n\right\}_{n\in\mathbb{N}}\subseteq B_Y$
 with non-convergent subsequence. But $\ell^1$-Theorem  and the assumption that
$\ell^1$ does not embed in $X$,
the sequence $\left\{s_n\right\}_{n\in\mathbb{N}}$ has a subsequence $\left\{s_{n_k}\right\}_{k\in\mathbb{N}}$
which is weak-Cauchy. So the sequence $\left\{s_{n_{k+1}}-s_{n_k}\right\}_{k\in\mathbb{N}}$
is weakly null. Moreover, since $\left\{s_{n_k}\right\}_{k\in\mathbb{N}}$ is not norm-convergent, there is $\theta>0$ such that the following holds:
\begin{equation}\label{2}
\forall n\in\mathbb{N}\quad\exists k,m\in\mathbb{N}:n<k
<m \quad\mbox{and}\quad||s_k-s_m||\geq\theta .
\end{equation}
By \eqref{2} we determine an increasing sequence $\left\{p_n\right\}_{n\in\mathbb{N}}
\subseteq\mathbb{N}$ such that
\begin{equation*}
\displaystyle||s_{p_{2n}}-s_{p_{2n-1}}||\geq\theta,\quad\forall n\in\mathbb{N}.
\end{equation*}
Defining $u_n=s_{p_{2n}}-s_{p_{2n-1}}$ we get
$u_n\overset{w}{\longrightarrow}0$ and $||u_n||\geq\theta,\quad\forall
n\in\mathbb{N}$.
Hence, the sequence $\left\{z_n\right\}_{n\in\mathbb{N}}$ defined by $z_n=\|u_n\|^{-1}u_n$
is unitary and weakly null.\end{proof}

\section{Definition of $\mathcal{JT}$ and non-separability of the conjugate}
In this section, we define the $\mathcal{JT}$ space and summarize some of its basic properties. As mentioned in the introduction, $\mathcal{JT}$ is the first example of a non-separable Banach space which does not contain $\ell^1$. In this section, we give the basic definitions about $\mathcal{JT}$ and show that the conjugate space is non-separable.

Let us begin with some basic definitions on the Cantor tree, which be the natural index set to define our basis. These definitions will turn out to be very important and will constantly be used in the following. 

\begin{itemize}
\item We define the Cantor tree as follows:$$2^{<\mathbb{N}}=\left\{s=(s_1,...,s_n):s_i\in\left\{0,1\right\},\quad\forall i=1,...,n,\quad n\in\mathbb{N}\right\}\cup\left\{\emptyset\right\}.$$  The empty set $\emptyset$ is called root of the tree. Elements of $2^{<\mathbb{N}}$ are called nodes of the tree.
\item We also define the set of sequences of $0$ and $1$ as follows:
$$2^{\mathbb{N}}=\left\{\sigma=(\sigma_n)_{n\in\mathbb{N}}:\sigma_i\in\left\{0,1\right\},\quad
\forall i\in\mathbb{N}\right\}.$$
\item We define the level function $|\cdot|:2^{<\mathbb{N}}\rightarrow \mathbb{N}\cup\left\{0\right\}$ by 
\begin{equation*}
|s| =\begin{cases}0,\quad s=\emptyset,\\
n,\quad s=(s_1,...,s_n).
\end{cases}
\end{equation*}
\item We define the partial ordering '$\sqsubseteq$' on $2^{<\mathbb{N}}$ as follows:
\begin{itemize}
\item $\emptyset\sqsubseteq s,\quad\forall s\in 2^{<\mathbb{N}}$
\item If $s,u\in 2^{<\mathbb{N}}$with $s,u\neq\emptyset$,
then $s\sqsubseteq t\Leftrightarrow |s|\leq|t|$ and $s_i=t_i,\quad\forall i=1,...,|s|.$
\end{itemize}
\item  For $\sigma\in 2^{\mathbb{N}}$ and $n\in\mathbb{N}$, we will denote ${\sigma|_n}=(\sigma_1,...,\sigma_n)\in 2^{<\mathbb{N}}$. If we consider pairwise distinct $\sigma^1,...,\sigma^n\in 2^\mathbb{N}$, then there is $N\in\mathbb{N}$ such that ${\sigma^i}|_N\neq {\sigma^j}|_N\quad\forall i,j\in\left\{1,...,n\right\}$ with $i\neq j$. The minimum positive integer with this property is called separation level of $\sigma^1,...,\sigma^n$.
\item Let $I\subseteq 2^{<\mathbb{N}}$ such that for any $s,t\in I$ we have either $s\sqsubseteq t$ or $t\sqsubseteq s$. If for any $s,t\in I$ and $w\in 2^{<\mathbb{N}}$ such that $s\sqsubseteq w\sqsubseteq t$, we have that $w\in I$, then $I$ is called an interval.  Finite intervals are called segments and are typically denoted by $F$, while infinite intervals are called branches and are typically denoted by $B$.
\item A segment $F$ can be uniquely written as $F=\left\{s_1,s_2,...,s_n\right\}$ where $s_i\in 2^{<\mathbb{N}},\quad\forall i=1,...,n$  and $s_1\sqsubseteq s_2 \sqsubseteq,...,\sqsubseteq s_n$.
Node $s_1$ is called  initial node of $F$ and is denoted as $\inode (F)$, while $s_n$ is called  ending node of $F$ and is denoted as $\enode (F)$. The nodes $\inode (F)$ and $\enode (F)$ are called endpoints of $F$. It is clear that for any $s,t\in 2^{<\mathbb{N}}$ with $s\sqsubseteq t$, there is segment $F$ with $\inode (F)=s$ and $\enode (F)=t$. 
 \item For any branch $B$ there is unique $\sigma(B)\in 2^{\mathbb{N}}$ and unique $n\in\mathbb{N}$ such that $B=\left\{(\sigma(B)_{n+k})_{k=0}^\infty\right\}$. We define the initial node of $B$ as $\inode (B)=\sigma(B)_n$.
 \item Let $B_1,...,B_n$ be pairwise distinct branches. Then $\sigma(B_1),...,\sigma(B_n)$ are clearly pairwise distinct too. We define the separation level of $B_1,...,B_n$ as the minimum positive integer $N$ such that ${\sigma(B_i)}|_N\neq {\sigma(B_j)}|_N,\quad\forall i,j\in\left\{1,...,n\right\}$ with $i\neq j$ and $N\geq \inode(B_i),\quad\forall i=1,...,n$.
\item We define $t_1=\emptyset$ and $t_2=(0)$. For $n>2$ we consider 
$t_n=(s_1,...,s_m)$. If for all $i=1,...m$ we have $s_i=1$, we define
$t_{n+1}=(s_1',...,s_m',s_{m+1}')$ where $s_i'=0$ for all $i=1,...,m+1$. Alternatively we consider $i_0=\max\left\{i\in\left\{1,...,m\right\}:s_i=0\right\}$ and define $t_{n+1}=(s_1',...,s_m')$ with 
$s_i'=s_i,\quad \forall i=1,...,i_0-1$, $s_{i_0}'=1$ and $s_i'=0,\quad\forall i=i_0+1,...,m$. By this enumeration, it is clear that $2^{<\mathbb{N}}=\left\{t_n:n\in\mathbb{N}\right\}$ and
that $2^{<\mathbb{N}}$ is countable.
\item We observe that for any $n\in\mathbb{N}$, we have that $|t_n|=[\log_2 n]$ where $[\cdot]$ denotes the integer part of a positive number.
\item For $s\in 2^{<\mathbb{N}}$, we denote $e_s=\mathcal{X}_{\left\{s\right\}}$.
In particular we denote $e_n=\mathcal{X}_{\left\{t_n\right\}}$. Clearly the sequences
$\left\{e_s\right\}_{s\in 2^{<\mathbb{N}}}$ and $\left\{e_n\right\}_{n\in\mathbb{N}}$ coincide and the represent the sequence of the characteristic functions of the nodes of the Cantor tree.
\end{itemize}

We are now in the position to define the James Tree Space.
\begin{definition}
We define
\begin{equation*}
\mathcal{JT}=\left\{x:2^{<\mathbb{N}}\rightarrow\mathbb{R}:\sup\left\{\sum_{i=1}^m|\sum_{s\in F_i}
x_s|^2
\right\}<\infty\right\},
\end{equation*}
where $\sup$ is taken over all finite families of pairwise disjoint segments
$\left\{F_i\right\}_{i=1}^m$, $m\in\mathbb{N}$. It is immediate that $\mathcal{JT}$ is infinite dimensional vector space.
For $x\in\mathcal{JT}$, we define $$||x||=\sup\left\{\displaystyle\sum_{i=1}^m|\sum_{s\in F_i}
x_s|^2\right\}^{1/2},$$  where $\sup$ is taken over all finite families of pairwise disjoint  segments  $\left\{F_i\right\}_{i=1}^m$.
\end{definition}
\begin{remark} Given $x\in\mathcal{JT}$, we define its support as
$$\supp(x)=\left\{s\in 2^{<\mathbb{N}}:x_s\neq 0\right\}.$$
It is clear that the supremum taken for $\|x\|$ can be restricted to all finite families of pairwise disjoint segments contained in $\supp (x)$.
\end{remark}
\begin{proposition}
$(\mathcal{JT},||\cdot||)$ is a Banach space.
\end{proposition}
\begin{proof}
 We first show that $||\cdot||$ is a norm. The only non-trivial part is triangular inequality. Consider $x,y\in\mathcal{JT}$ and 
 $\left\{F_i\right\}_{i=1}^m$  pairwise disjoint segments. Then Minkowski's inequality for sums implies
 \begin{align*}
 (\sum_{i=1}^m|\sum_{s\in F_i}(x_s+y_s)|^2)^{1/2}&=[(\sum_{s\in F_1}x_s+\sum_{s\in F_1}y_s)^2+\\
 &\hspace{0.5cm}+...+(\sum_{s\in F_m}x_s+\sum_{s\in F_m}y_s)^2]^{1/2}\\
 &\leq(\sum_{i=1}^m|\sum_{s\in F_i}x_s|^2)^{1/2}+
 (\sum_{i=1}^m|\sum_{s\in F_i}y_s|^2)^{1/2}\\
 &\leq ||x||+||y||,
 \end{align*}
 so, taking supremum, since the intervals chosen are arbitrary, we obtain
 $$||x+y||\leq ||x||+||y||.$$
  Let us now show completeness.
Let $\left\{x_n\right\}_{n\in\mathbb{N}}$ be a Cauchy sequence in $\mathcal{JT}$. Then for any 
$\epsilon>0$, there is $N$ such that for any $k>n\geq N$, there holds $||x_k-x_n||<\epsilon$.
Considering $s\in 2^{<\mathbb{N}}$ and the interval $I_s=\left\{s\right\}$, the definition of the norm implies
that $$|x_{k,s}-x_{n,s}|<\epsilon,\quad\forall k>n>N.$$ Therefore, the sequence
$\left\{x_{n,s}\right\}_{n\in\mathbb{N}}$ converges for any $s\in 2^{<\mathbb{N}}$. Let us define the map
$x:2^{<\mathbb{N}}\rightarrow\mathbb{R}$, given by
$$x_s=\lim_{n\to\infty} x_{n,s},\quad s\in 2^{<\mathbb{N}}.$$
We will show $x\in\mathcal{JT}$ and that $\displaystyle x=\lim_{n\to\infty} x_n$.
Indeed, consider $\epsilon>0$ and $M\in\mathbb{N}$ such that for any $k>n\geq M$, we have
$||x_n-x_k||<\displaystyle\frac{\epsilon}{2}$. Then for any  family of pairwise disjoint segments
 $\left\{F_i\right\}_{i=1}^m$, we obtain
\begin{eqnarray}\label{10}
(\sum_{i=1}^m|\sum_{s\in F_i}|x_{m,s}-x_{n,s})|^2)^{1/2}&<&\frac{\epsilon}{2}
,\quad\forall m>n\geq M\overset{m\to\infty}{\Rightarrow}\\
\label{11}
(\sum_{i=1}^m|\sum_{s\in F_i}|x_{n,s}-x_s||^2)^{1/2}&\leq&\frac{\epsilon}{2}<\epsilon
,\quad\forall n\geq M.
\end{eqnarray}
Fixing $n=M$ in \eqref{10}, we get that $x_M-x\in\mathcal{JT}\Rightarrow x\in\mathcal{JT}$. Then \eqref{11} yields $$ x=\lim_{n\to\infty} x_n.$$
\end{proof}

The following Lemma illustrates an interesting super-additive property of the norm which will yield that $\{e_n\}_{n\in\mathbb{N}}$ is a Schauder basis.
 \begin{lemma}\label{lemma estimates}
 Let $x\in\mathcal{JT}$. Consider $k\in\mathbb{N}$  and $s_k=\sum_{i=1}^k x_i(t_i)e_i$. Then the following estimate holds:
 \begin{equation*}
 \|s_k\|^2+\|x-s_k\|^2\leq \|x\|^2.
 \end{equation*}
\end{lemma} 
 \begin{proof}
Fix $k\in\mathbb{N}$. Then it is clear that 
\begin{equation}\label{estimates}
\begin{aligned}
\|s_k\|^2&=\sup\left\{\sum_{i=1}^m|\sum_{s\in F_i}x_s|^2\right\},\\
\|x-s_k\|^2&=\sup\left\{\sum_{j=1}^l|\sum_{s\in F_j'}x_s|^2\right\},
\end{aligned}
\end{equation}
where the supremums are taken over all finite families of pairwise disjoint segments $\{F_i\}_{i=1}^m$, $\{F_j\}_{j=1}^l$ contained in $\{t_1,...,t_k\}$ and $\{t_{n}:n\geq k+1\}$ respectively. Considering such families, we define the family of disjoint intervals $\mathcal{F}=\{F_1,...,F_m,F_1',...,F_l'\}$. Then we clearly have
$$\sum_{i=1}^m|\sum_{s\in F_i}x_s|^2+\sum_{j=1}^l|\sum_{s\in F_j'}x_s|^2=\sum_{F\in\mathcal{F}}|\sum_{s\in F}x_s|^2\leq\|x\|^2.$$
Since the families are chosen arbitralily, we may take supremums and obtain the required estimate from \eqref{estimates}.
 \end{proof}
 \begin{proposition}
 The sequence  $\left\{e_n\right\}_{n\in\mathbb{N}}$ 
 is monotone and unitary Schauder basis of $\mathcal{JT}$.
 \end{proposition}
 \begin{proof}
 Clearly $\quad e_n\neq 0\quad\forall n\in\mathbb{N}$ and $||e_n||=1\quad\forall n\in\mathbb{N}$. We first show that
 $\mathcal{JT}=[e_n:n\in\mathbb{N}]$. Consider $x\in\mathcal{JT}$. Let us define $s_n=\displaystyle\sum_{i=1}^n x(t_i)e_i$. We will show that $x=\displaystyle\lim_{n\to\infty} s_n.$ Indeed, let $\epsilon>0$. The definition of the norm yiels that there are
 pairwise disjoint segments $\left\{F_i\right\}_{i=1}^m$ 
such that,
\begin{equation}\label{11}
\sum_{i=1}^m|\sum_{s\in F_i}x_s|^2>||x||^2-\epsilon^2.
\end{equation}
For $n_0=\max\left\{|s|:s\in\displaystyle\bigcup_{i=1}^m F_i\right\}$, inequality
 \eqref{11} yields that
\begin{equation*}
||x||^2-||s_n||^2<\epsilon^2\quad\forall n\geq 2^{n_0+1}.
\end{equation*}
So for $n\geq 2^{n_0+1}$, Lemma \ref{lemma estimates} implies that $$||x-s_n||^2=||x-s_n||^2+||s_n||^2-||s_n||^2=
||x||^2-||s_n||^2<\epsilon^2.$$
Therefore, $x=\displaystyle\lim_{n\to\infty} s_n.$
Finally for $n\in\mathbb{N}$ and $\lambda_1,...,\lambda_n,\alpha_{n+1}\in\mathbb{R}$,
we consider pairwise disjoint segments $\left\{F_i\right\}_{i=1}^m$ with $t_{n+1}\notin\displaystyle\bigcup_{i=1}^m F_i$.
Then, defining $y=\sum_{i=1}^n\lambda_i e_i$, we obtain
\begin{eqnarray*}
(\sum_{i=1}^m|\sum_{s\in F_i}y_s|^2))^{1/2}\leq 
||\sum_{i=1}^{n+1} \lambda_ie_i||\Rightarrow
||\sum_{i=1}^n\lambda_ie_i||\leq ||\sum_{i=1}^{n+1} \lambda_ie_i||,
\end{eqnarray*}
and the claim is proved.
\end{proof}

\begin{remark}
 $\mathcal{JT}$ can be equivalently defined as the completion of $<c_{00}(2^{<\mathbb{N}})>$ under the norm defined.
\end{remark}
\begin{remark}\label{intervals}
If $F$ is a segment, we define $F^*=\sum_{s\in F} e_s^*$. Clearly $F\in\mathcal{JT}^*$ since $e_s\in\mathcal{JT}^*,\quad\forall s\in 2^{<\mathbb{N}}$.

Consider now a branch $B$. The definition of the norm implies that for any $x\in\mathcal{JT}$ and $\epsilon>0$, there is $n_0\in\mathbb{N}$ such that for any $m>n>n_0$, there holds $$|\sum_{i=n+1}^m e_i^*(x)|< \epsilon.$$ 
Therefore, the series $\sum_{s\in B}e_s^*(x)$ converges for any $x\in\mathcal{JT}$. Defining $B^*:\mathcal{JT}\rightarrow\mathbb{R}$ by $$B^*(x)=\displaystyle\sum_{s\in B}e_s^*(x),\quad\forall x\in\mathcal{JT},$$
 we obtain that $B^*$ is linear and, by Banach-Steinhauss Theorem, we get $B^*\in\mathcal{JT}^*$. It is clear that $$B^* \overset{w^*}{=}\displaystyle\sum_{s\in B}e_s^*.$$
 
 Finally it is clear that $||I^*||=1$, for any interval $I$, and that for any $x\in\mathcal{JT}$, we have the following norm description:
\begin{equation*}
||x||=\displaystyle\sup\left\{\sum_{i=1}^m |I_i^*(x)|^2\right\}^{1/2},
\end{equation*}
where $\left\{I_i\right\}_{i=1}^m$ are pairwise disjoint intervals, which can be either segments or branches.
\end{remark}

Let us introduce some notation. Recall that for $\sigma\in 2^{\mathbb{N}}$, we denote $\sigma|_n=(\sigma_1,...,\sigma_n)\in 2^{<\mathbb{N}}$. We define
$$\sigma^*\overset{w^*}{=}\sum_{n=1^\infty}e_{\sigma|_n}^*\in\mathcal{JT}^*,\quad\text{by Remark \ref{intervals}}.$$
We can now easily see that the conjugate space is non-separable.
\begin{proposition}
The conjugate space $\mathcal{JT^*}$ is non-separable.
\end{proposition}
\begin{proof}
The set $\left\{\sigma^* :\sigma\in 2^{\mathbb{N}}\right\}$ is clearly 
uncountable and $1$-separated i.e. there are $x^*,y^*\in\mathcal{JT}^*$ with
$$\|x^*-y^*\|\geq 1.$$
 Indeed, considering $\sigma_1\neq\sigma_2\in 2^{\mathbb{N}}$, there is $s\in 2^{<\mathbb{N}}$ with
$\sigma_1^*(e_s)=1$ and $\sigma_2^*(e_s)=0$. Thus $||\sigma_1^*-\sigma_2^*||\geq 1$,
so $\mathcal{JT^*}$ is not separable.
\end{proof}

\section{ The $\ell^1$ non-embedding in $\mathcal{JT}$}
 The goal of this section is to show that $\ell^1$ does not embed in $\mathcal{JT}$. As mentioned before, we will use Riesz's Representation Theorem to prove the non-embedding of $\ell^1$.

Let $\left\{I_n\right\}_{n\in\mathbb{N}}$ be pairwise disjoint intervals. For any $x\in\mathcal{JT}$ we have that $$\sum_{n=1}^\infty|I_n^*(x)^2|\leq ||x||^2.$$
Moreover, for any sequence $\left\{\alpha_n\right\}_{n\in\mathbb{N}}\in B_{\ell^2}$, Cauchy-Schwartz inequality implies that 
$$\sum_{n=1}^\infty|\alpha_n||I_n^*(x)|\leq ||x||,$$ 
so $\sum_{n=1}^\infty\alpha_nI_n^*(x)$ converges absolutely for any $x\in\mathcal{JT}$ and $w^*-\sum_{n=1}^\infty\alpha_nI_n^*\in B_{\mathcal{JT^*}}$. 
Let us note that by $w^*-\sum_{n=1}^\infty\alpha_nI_n^*$ we mean the series interpreted as a $w^*$-limit of partial sums.

Let us define  
\begin{equation*}
K^*=\left\{w^*-\displaystyle\sum_{n=1}^\infty\alpha_nI_n^*:\left\{\alpha_n\right\}_{n\in\mathbb{N}}\in B_{\ell^2}\mbox{ and } \left\{I_n\right\}_{n\in\mathbb{N}}\mbox{ pairwise disjoint intervals}\right\}.
\end{equation*}
Clearly $K^*\subseteq B_{\mathcal{JT^*}}$.
\begin{proposition}\label{norming}
 The set $K^*$ is norming for $\mathcal{JT}$ i.e. $$||x||=\sup\left\{k^*(x):k^*\in K^*\right\}, \quad\forall x\in\mathcal{JT}.$$ We may also write $||x||=\sup\left\{|k^*(x)|:k^*\in K^*\right\}$.
\end{proposition} 

\begin{proof}
Let $x\in\mathcal{JT}$. Since $K^*\subseteq B_{\mathcal{JT^*}}$, we have that
$$\sup\left\{k^*(x):k^*\in K^*\right\}\leq||x||.$$ For the opposite direction, consider $n\in\mathbb{N}$
and define $x_n=\displaystyle\sum_{i=1}^ne_i^*(x)e_i$. Then there are pairwise disjoint intervals $\left\{I_i\right\}_{i=1}^m$
 such that $$||x_n||-\frac{1}{n}<\displaystyle(\sum_{i=1}^m|I_i^*(x_n)|^2)^{1/2}.$$ 
 Defining $$\lambda_i=(\sum_{i=1}^m|I_i^*(x_n)|^2)^{-1/2}I_i^*(x_n),$$
we have that $k^*=\displaystyle\sum_{i=1}^m\lambda_iI_i^*\in K^*$ and $k^*(x_n)=(\displaystyle\sum_{i=1}^m|I_i^*(x_n)|^2)^{1/2}$,
so $$||x_n||-\frac{1}{n}<k^*(x_n)\leq\sup\left\{k^*(x_n):k^*\in K^*\right\}\overset{n\to\infty}{\Rightarrow} ||x||\leq\sup\left\{k^*(x):k^*\in K^*\right\},$$ and the result follows. The second description is immediate.
\end{proof}
\begin{proposition}
The set $K^*$ is  $w^*$-compact subset of $\mathcal{JT^*}$.
\end{proposition}
\begin{proof}
Since $\mathcal{JT}$ is separable, the ball $(B_{X^*},w^*)$ is a metric space, so it suffices to show that $K^*$ is sequentially compact. Let $k_n^*=w^*-\displaystyle\sum_{n=1}^\infty\alpha_{i,n} I_{i,n}^*$ a sequence in $K^*$. Since this series converges absolutely, we may assume, after a possible re-ordering, that for any $n\in\mathbb{N}$, we have that $|\alpha_{i+1,n}|\leq|\alpha_{i,n}|\quad\forall i\in\mathbb{N}$. We first prove a claim.

$\bullet$ \textit{Claim}:
Let $\left\{I_n\right\}_{n\in\mathbb{N}}$ a sequence of intervals. Then there is subsequence
 $\left\{I_{k_n}\right\}_{n\in\mathbb{N}}$ and interval $I$ such that $I_{k_n}^*\overset{w^*}{\longrightarrow} I^*$.
 
\textit{Proof of the claim} Using a diagonal argument, we may find subsequence $\left\{I_{k_n}\right\}_{n\in\mathbb{N}}$  such that $\left\{I_{k_n}^*(e_s)\right\}_{n\in\mathbb{N}}$ converges for any $s\in 2^{<\mathbb{N}}$. Let us define $$I=\left\{s\in 2^{<\mathbb{N}}:\exists n_s\in\mathbb{N}\mbox{ with }s\in I_{k_n}\quad\forall n\geq n_s\right\}.$$ Then $I$ is an interval. Indeed, consider $s,t\in I$. Then there is $n_0\in\mathbb{N}$ such that $s,t\in I_{k_{n_0}}$. Thus, either $s\sqsubseteq t$ or $t\sqsubseteq s$. Let us consider $s,t\in I$ and $w\in 2^{<\mathbb{N}}$ such that $s\sqsubseteq w\sqsubseteq t$. From the definition of $I$, there is $n_0\in\mathbb{N}$ such that $s,t\in I_n,\quad\forall n\geq n_0$. Hence, we have that $w\in I_n,\quad\forall n\geq n_0\Rightarrow w\in I$. We may easily show that $\displaystyle\lim_{n\to\infty}I_{k_n}^*(s)=I^*(s),\quad\forall s\in 2^{<\mathbb{N}}$ and since $||I_n^*||=1,\quad\forall n\in\mathbb{N}$, Corollary \ref{cor 1.1.11}, implies the claim.

\textit{Main proof}
Using a diagonal argument, we may find $M\in[\mathbb{N}]$ , sequence $\left\{\alpha_i\right\}_{i\in\mathbb{N}}\in B_{\ell^2}$ such that $\alpha_{i,n}\overset{n\in M}{\longrightarrow}\alpha_i,\quad\forall i\in\mathbb{N}$ and intervals $\left\{I_i\right\}_{i\in\mathbb{N}}$ such that $I_i=w^*-\displaystyle\lim_{n\in M} I_{i,n}^*$. The intervals $\left\{I_i\right\}_{i\in\mathbb{N}}$ are clearly pairwise disjoint.

 We define
$k^*\overset{w^*}{=}\sum_{n=1}^\infty\alpha_iI_i^*\in K^*$ and we will show that
$$k^*\overset{w^*}{=}\displaystyle\lim_{n\in M}k_n^*.$$ Indeed, consider $s\in 2^{<\mathbb{N}}$ and $\epsilon>0$. We pick $N\in\mathbb{N}$ such that $$(\displaystyle\sum_{i=N+1}^\infty\alpha_i^2)^{1/2}<\frac{\epsilon}{4}.$$ Then there is $n_0\in\mathbb{N}$ such that for any $n\geq n_0$, there holds $$\displaystyle\sum_{i=1}^N|\alpha_{i,n}I_{i,n}^*(e_s)-\alpha_iI_i^*(e_s)|<\frac{\epsilon}{4},$$ and $$\displaystyle |\alpha_{N+1,n}-\alpha_{N+1}|<\frac{\epsilon}{4}.$$
Then for any $n\in M$ with $n\geq n_0$, we have
\begin{align*}
&\hspace{-1cm}|\sum_{i=1}^\infty\alpha_{i,n}I_{i,n}^*(e_s)-\sum_{i=1}^\infty\alpha_iI_i^*(e_s)|\leq\sum_{i=1}^N|\alpha_{i,n}I_{i,n}^*(e_s)-\alpha_iI_i^*(e_s)|\\
&+|\sum_{i=N+1}^\infty\alpha_{i,n}I_{i,n}^*(e_s)|+
|\sum_{i=N+1}^\infty\alpha_iI_i^*(e_s)|\\
&<\frac{\epsilon}{4}+(\sum_{i=N+1}^\infty\alpha_i^2)^{1/2}(\sum_{i=N+1}^\infty|I_i^*(e_s)|^2)^{1/2}\\
&+|\alpha_{j,n}|\quad\mbox{, for some}\quad j\geq N+1\\
&< \frac{\epsilon}{4}+\frac{\epsilon}{4}+|\alpha_{N+1,n}|\\
&\leq \frac{\epsilon}{2}+|\alpha_{N+1,n}-\alpha_{N+1}|+|\alpha_{N+1}|\\
&\leq\frac{\epsilon}{2}+\frac{\epsilon}{4}+\frac{\epsilon}{4}=\epsilon.
\end{align*}
Therefore $k_n^*(e_s)\overset{n\in M}{\longrightarrow}k^*(e_s)\quad\forall s\in 2^{<\mathbb{N}}$ and the result is proved.\end{proof}

We will now use some measure theoretic arguments. Let $(X,\mathcal{A},\mu)$ a signed measure space and $f\in L^1(|\mu|)$, where $|\mu|$ is the total variation of $\mu$. The integral of $f$ with respect to $\mu$ is defined as $$\displaystyle\int_{X}\!f\, d\mu=\displaystyle\int_{X}\!f\, d\mu^+-\displaystyle\int_{X}\!f\, d\mu^-.$$
Dominated Convergence Theorem implies that if we consider a sequence of measurable functions $f_n:(X,\mathcal{A})\rightarrow\mathbb{R}$ with $f=\displaystyle\lim_{n\to\infty} f_n$ a.e., such that there is $g\in L^1(|\mu|)$ with $$|f_n|\leq g\quad\forall n\in\mathbb{N} \text{ a.e. in } X,$$ then
$$\displaystyle\lim_{n\to\infty}\int_{X}\!|f_n-f|\, d\mu=0\Rightarrow
\lim_{n\to\infty}\int_{X}\!f_n\, d\mu=\int_{X}\!f \, d\mu.$$
We will now us a special form of Riesz's Representation Theorem, whose proof can be found in \cite{rudin}. Let $X$ be a topological space. We will write $\mathcal{M}_{f,r}(X)$ for the set of all finite, regular, signed Borel measures of $X$ and $C(X)$ for the Banach space of continuous real functions on $X$, equipped with the $\|\cdot\|_\infty$ norm.
\begin{theorem}(Riesz's Representation Theorem)
Let $X$ be a compact Hausdorff space. Then for any $f^*\in C^*(X)$ there is unique $\mu_{f^*}\in\mathcal{M}_{f,r}(X)$ such that $$f^*(f)=\displaystyle\int_{X}\!f \, d\mu_{f^*},\quad\forall f\in C(X).$$
\end{theorem}

\begin{proposition}\label{prop 3.2.4}
Let $\left\{x_n\right\}_{n\in\mathbb{N}}$ be a bounded sequence in $\mathcal{JT}$. If the sequence $\left\{I^*(x_n)\right\}_{n\in\mathbb{N}}$ converges for any interval $I$ then $\left\{x_n\right\}_{n\in\mathbb{N}}$ is $w$-Cauchy.
\end{proposition}
\begin{proof}
Let $M=\displaystyle\sup_{n\in\mathbb{N}}\left\{||x_n||\right\}$. We define $T:\mathcal{JT}\rightarrow C(K)$ by $T(x)=\widehat{x}|_K$. Clearly $T$ is a well-defined linear isometry. Indeed,
$$||T(x)||_\infty=||\widehat{x}|_K||_\infty=\displaystyle\sup\left\{|\widehat{x}(x^*)|:x^*\in K^*\right\}=
\displaystyle\sup\left\{|x^*(x)|:x^*\in K^*\right\}=||x||,$$ 
by Proposition \ref{norming}. Therefore the conjugate map $T^*:C^*(K)\rightarrow\mathcal{JT^*}$ is surjective. So, for any $x^*\in\mathcal{JT^*}$, there is $f^*\in C^*(K)$ such that $x^*=f^*\circ T$. Riesz's Representation Theorem implies that for any $x^*\in\mathcal{JT^*}$, there is unique $\mu_{x^*}\in\mathcal{M}_{f,r}(K)$ such that for any $n\in\mathbb{N}$, we have
$$x^*(x_n)=\displaystyle\int_{K}\!\widehat{x}_n|_K, d\mu_{x^*}.$$ 
Since $||\widehat{x}_n|_K||\leq M$ for any $n\in\mathbb{N}$ and $|\mu_{x^*}|(K)<\infty$ for any $x^*\in\mathcal{JT^*}$, Dominated Convergence Theorem yields it is enough to show that $\left\{x^*(x_n)\right\}_{n\in\mathbb{N}}$ converges for all $x^*\in K^*$. Indeed, let $x^*\overset{w^*}{=}\sum_{i=1}^\infty\lambda_iI_i^*\in K^*$. By assumption we may write $$\alpha_i=\displaystyle\lim_{n\to\infty}I_i^*(x_n).$$ Then
$(\sum_{i=1}^\infty\alpha_i^2)^{1/2}\leq M$ and $\sum_{i=1}^\infty|\lambda_i\alpha_i|\leq M$, by Cauchy-Schwartz inequality. Hence, the series $\sum_{i=1}^\infty\lambda_i\alpha_i$ is absolutely convegent. Consider $\epsilon>0$
and $N\in\mathbb{N}$ such that $$(\displaystyle\sum_{i=N+1}^\infty\lambda_i^2)^{1/2}<\frac{\epsilon}{3M},$$ and $$|\displaystyle\sum_{i=N+1}^\infty\lambda_i\alpha_i|<\frac{\epsilon}{3}.$$ Fixing $n_0\in\mathbb{N}$ such that for any $n\geq n_0$ there holds $$\displaystyle\sum_{i=1}^N|\lambda_iI_i^*(x_n)-\lambda_i\alpha_i|<\frac{\epsilon}{3},\quad\forall n\geq n_0,$$
we obtain
\begin{align*}
|x^*(x_n)-\sum_{i=1}^\infty\lambda_i\alpha_i|&\leq\sum_{i=1}^N|\lambda_iI_i^*(x_n)-\lambda_i\alpha_i|+|\sum_{i=N+1}^\infty\lambda_iI_i^*(x_n)|+|\sum_{i=N+1}^\infty\lambda_i\alpha_i|\\
\hspace{-1cm}&<\frac{\epsilon}{3}+(\sum_{i=N+1}^\infty\lambda_i^2)^{1/2}(\sum_{i=N+1}^\infty |I_i^*(x_n)|^2)^{1/2}+\frac{\epsilon}{3}\\
&<\frac{\epsilon}{3}+\frac{\epsilon}{3}+\frac{\epsilon}{3}=\epsilon.
\end{align*}
The proof is complete.
\end{proof}

We will use the following notation. Given  a countable set $M$, we will write $[M]$ to denote the set of infinite subsets of $M$. Let us first prove the following Lemma.
\begin{lemma}\label{lem fuc}
Let $X\neq\emptyset$ and $f_n:X\rightarrow\mathbb{R}$ sequence of functions such that
$$\displaystyle\sup_{n\in\mathbb{N}}\left\{|f_n(x)|\right\}<\infty\quad\forall x\in X.$$ If for any $\epsilon>0$ and $M\in[\mathbb{N}]$ there is $L\in [M]$ with
$$\displaystyle\limsup_{n\in L}f_n(x)-\displaystyle\liminf_{n\in L}f_n(x)<\epsilon,\quad\forall x\in X,$$  the sequence $\left\{f_n\right\}_{n\in\mathbb{N}}$ has pointwise subsequence.
\end{lemma}
\begin{proof}
By induction we will construct a decreasing sequence $\left\{L_k\right\}_{k\in\mathbb{N}}$ of infinite subsets of
$\mathbb{N}$ such that for any $k\in\mathbb{N}$ we have $$\displaystyle\limsup_{n\in L_k}f_n(x)-\displaystyle\liminf_{n\in L_k}f_n(x)<\frac{1}{k},\quad\forall x\in X.$$ Consider a strictly increasing sequence
 $\left\{n_k\right\}_{k\in\mathbb{N}}$ such that $n_k\in L_k\quad\forall k\in\mathbb{N}$ and the set $L_\infty=\left\{n_k:k\in\mathbb{N}\right\}$. Clearly $L_\infty$ is infinite and for any $k\in\mathbb{N}$ we have $L_\infty\subseteq L_k$. Then for all $x\in X$ we have that
\begin{equation*}
\displaystyle\limsup_{n\in L_\infty}f_n(x)-\displaystyle\liminf_{n\in L_\infty}f_n(x)\leq \displaystyle\limsup_{n\in L_k}f_n(x)-\displaystyle\liminf_{n\in L_k}f_n(x)<\frac{1}{k},\quad\forall k\in\mathbb{N},
\end{equation*} 
so letting $k\to\infty$, we obtain
$$\limsup_{n\in L_\infty}f_n(x)=\liminf_{n\in L_\infty}f_n(x),$$
thus the sequence $(f_n)_{n\in L_\infty}$ is pointwise convergent.
\end{proof}

We are now able to prove the non-embedding of $\ell^1$ in $\mathcal{JT}$.

\begin{theorem}
\ $\ell^1$ does not embed in $\mathcal{JT}$.
\end{theorem}
\begin{proof}
Assume $\ell^1$ embeds in $\mathcal{JT}$. Then
Proposition \ref{prop 1.4.5} implies there is block of the basis of $\mathcal{JT}$ equivalent to the standard basis of $\ell^1$.
We will show that any block has w-Cauchy subsequence, which contradicts the $\ell^1$-Theorem. Let $\left\{u_n\right\}_{n\in\mathbb{N}}$ a block and let us write $M=\displaystyle\sup_{n\in\mathbb{N}}\left\{||u_n||\right\}$. 
We will show there is subsequence $\left\{u_{n_k}\right\}_{k\in\mathbb{N}}$ such that $\left\{I^*(u_{n_k}\right\}_{k\in\mathbb{N}}$ converges for any interval $I$ and the contradiction will come by Proposition \ref{prop 3.2.4}.
For segments, using a diagonal argument, we may find $K\in[\mathbb{N}]$ such that $\left\{S^*(u_n)\right\}_{n\in K}$ converges for any segment $S$. Therefore, using  Lemma \ref{lem fuc}, it suffices to show that for any $\epsilon>0$ and $M\in[K]$, there is $L\in[M]$ such that $$\displaystyle\limsup_{n\in L}\sigma^*(u_n)-\displaystyle\liminf_{n\in L}\sigma^*(u_n)\leq\epsilon,\quad\forall \sigma\in 2^{\mathbb{N}}.$$ 
 Arguing by contradiction, assume there is $\epsilon>0$ and $M\in [K]$, such that for any $L\in[M]$, there is $\sigma_L\in 2^\mathbb{N}$ with
\begin{equation}\label{12}\limsup_{n\in L}\sigma_L^*(u_n)-\displaystyle\liminf_{n\in L}\sigma_L^*(u_n)>\epsilon\Rightarrow\displaystyle\limsup_{n\in L}{\sigma_L^*}^2(u_n)+\displaystyle\liminf_{n\in L}{\sigma_L^*}^2(u_n)>\frac{\epsilon^2}{2}.\end{equation} 
We pick $k\in\mathbb{N}$ such that $k\displaystyle\frac{\epsilon^2}{4}>M^2$. Consider $L_0\in[M]$
and $\sigma_1\in 2^\mathbb{N}$ such that \eqref{12} holds. Then at least one of
$\displaystyle\limsup_{n\in L_0}{\sigma_1^*}^2(u_n)$ and $\displaystyle\liminf_{n\in L_0}{\sigma_1^*}^2(u_n)$ is greater than $\displaystyle\frac{\epsilon^2}{4}$. Hence, there is $L_1\in[L_0]$ such that
$\displaystyle\lim_{n\in L_1}{\sigma_1^*}^2(u_n)>\frac{\epsilon^2}{4}$ and  $\sigma_2\in 2^\mathbb{N}$ which satisfies \eqref{12}. It is clear that $\sigma_1\neq \sigma_2$.
Continuing inductively we may find $L_k\subseteq L_{k-1}\subseteq,...,\subseteq L_1\in\mathbb{N}$ and $\sigma_1,...,\sigma_k\in 2^\mathbb{N}$ pairwise distinct such that
$\displaystyle\lim_{n\in L_i}{\sigma_i^*}^2(u_n)>\frac{\epsilon^2}{4}$, for any $i=1,...,k$, thus $ \displaystyle\lim_{n\in L_k}{\sigma_i^*}^2(u_n)>\frac{\epsilon^2}{4},\quad\forall i=1,...,k$. Therefore there is $N\in L_k$ such that for any $i=1,...,k$, we have
$${\sigma_i^*}^2(u_n)>\displaystyle\frac{\epsilon^2}{4},\quad\forall n\in L_k:n\geq N.$$
But $\sigma_1,...,\sigma_k$ finally separate, since they are pairwise distinct and the sequence $\left\{u_n\right\}_{n\in\mathbb{N}}$ is block, so we may find $n_0\in L_k$ with $n_0\geq N$ such that $$||u_{n_0}||^2\geq\displaystyle\sum_{i=1}^k{\sigma_i^*}^2(u_{n_0})>k\frac{\epsilon^2}{4}>M^2.$$ But this contradicts the boundness assumption on $\left\{u_n\right\}_{n\in\mathbb{N}}$. The proof is complete.
\end{proof}
\begin{corollary}
Every bounded sequence in $\mathcal{JT}$ has $w-$Cauchy subsequence.
\end{corollary}
\begin{proof}
It comes immediately by the fact that $\ell^1$ does not embed in $\mathcal{JT}$ and the $\ell^1$-Theorem.
\end{proof}

\end{document}